\documentclass[11pt]{article}

\usepackage{amsmath, amsthm, amssymb, graphicx, natbib}
\usepackage[margin = 1.2in]{geometry}

\usepackage{amssymb,natbib,amsmath}
\renewcommand{\epsilon}{\varepsilon}

\newtheorem{theorem}{Theorem}
\newtheorem{lemma}{Lemma}
\newtheorem{proposition}{Proposition}
\newtheorem{corollary}{Corollary}

\theoremstyle{remark}
\newtheorem{assumption}{Assumption}
\newtheorem{remark}{Remark}

\newcommand{\real}{\mathbb{R}}

\newcommand{\inte}{\mathbb{Z}}

\newcommand{\diam}{{\rm diam}}

\renewcommand{\S}{\mathcal{S}}

\newcommand{\dirp}{\textit{DP}}
\newcommand{\gam}{\textit{Ga}}

\newcommand{\scP}{\mathcal{P}}
\newcommand{\scQ}{\mathcal{Q}}

\newcommand{\sigmalo}{\underline{\sigma}}

\newcommand{\beq}{\begin{equation}}
\newcommand{\eeq}{\end{equation}}
\newcommand{\benum}{\begin{enumerate}}
\newcommand{\eenum}{\end{enumerate}}

\title{Adaptive Convergence Rates of a Dirichlet Process Mixture of Multivariate Normals}
\author{Surya T Tokdar \\ {\small \it Duke University}}
\date{}
 
\begin{document}

\maketitle
\begin{abstract}
It is shown that a simple Dirichlet process mixture of multivariate normals offers Bayesian density estimation with adaptive posterior convergence rates. Toward this, a novel sieve for  non-parametric mixture densities is explored, and its rate adaptability to various smoothness classes of densities in arbitrary dimension is demonstrated. This sieve construction is expected to offer a substantial technical advancement in studying Bayesian non-parametric mixture models based on stick-breaking priors. \\[5pt]
\end{abstract}

\section{Introduction}

Asymptotic frequentist properties of Bayesian non-parametric methods have received a lot of attention in recent years. 
It is now recognized that 
a single, fully automatic Bayesian model can offer adaptive, optimal rates of convergence for large collections of  {\it true} data generating distributions, ranging over several smoothness classes. In a seminal work, \citet{vandervaart&vanzanten09} establish adaptability of rescaled Gaussian process models for non-parametric regression, classification and density estimation. \cite{rousseau10} discusses adaptive density estimation with finite beta mixtures with a hierarchical prior on the number of mixture components. \citet{kruijer.10} and \cite{dejonge.zanten10} derive similar results for finite location-scale mixture models, respectively, in density estimation and regression, again with a prior on the number of mixture components.

Quite interestingly, adaptability has not yet been established for Dirichlet process (DP) mixture of normals models for density estimation. Even rates of convergence of these models remain to be derived beyond the univariate case. This is surprising because these models are the most studied of all Bayesian non-parametric models, and have been among the firsts for which positive results about convergence of the posterior were established \citep{ghosal99, ghosal&vandervaart01, ghosal&vandervaart07}. 

The main challenge in establishing adaptability of DP mixture models and to derive rates of convergence in higher dimensions lies in constructing a suitable {\it low-entropy, high-mass sieve} on the space of non-parametric mixture densities. Such sieve constructions are an integral part of the current technical machinery for deriving rates of convergence. The sieves that have been used to study DP mixture models \citep[e.g., in][]{ghosal&vandervaart07} do not scale to higher dimensions and lack adaptability to smoothness classes \citep{wughosal10}. 

The main import of this article is to plug this gap. It is demonstrated that a novel sieve construction proposed by this author \citep[reported earlier in an yet unpublished paper][]{pati.etal.11} give the desired dimension-scalability and smoothness-adaptability. This sieve utilizes the well known stick-breaking representation of a DP \citep{sethuraman94} and can be potentially useful for studying a large class of stick-breaking mixture models beyond the DP mixtures \citep[e.g.,][]{dunsonpark08, chungdunson09, rodriguez11}. 

This sieve paves way to the following results. For independent and identically distributed observations $X_1, \cdots, X_n$ from an unknown density  $p$ on $\real^d$, posterior convergence rates are derived for a simple DP location mixture model at a true data generating density $p_0$ belonging to either a class of infinitely differentiable densities or a class of compactly supported densities with two continuous derivatives. The derived rates are minimax optimal for these classes (up to logarithmic factors), and adapt to these two classes without requiring any user intervention to select or estimate any tuning parameters.

The two classes considered here form two extremes of the classes of smooth densities. Finer rate adaptability results can be derived by looking at the intermediate classes of H\"older smooth densities. These classes have well defined minimax optimal rates associated with them. It is demonstrated that the new sieve works for all H\"older classes. However, we stop short of deriving precise rates of convergence for these classes. This derivation requires an additional calculation of prior thickness rates for a $p_0$ belonging to these classes, which is a challenging and interesting problem but is tangential to the focus of this article. Interested readers are referred to some recent developments reported in \citet{kruijer.10}.

\section{A simple DP location mixture model}
\label{model}

Let $\phi_\sigma$ denote the density of the $d$-variate normal distribution with mean zero and variance $\sigma^2 I$. For any probability measure $F$ on $\real^d$, use $p_{F, \sigma}$ to denote the mixture density
\begin{equation}
\label{def mix}
p_{F,\sigma}(x) = \int \phi_\sigma(x - z) dF(z),\;\; x \in \real^d.
\end{equation}
Assign $p$ a prior distribution $\Pi$ given by the law of the random density $p_{F, \sigma}$ when $(F, \sigma^{-d}) \sim \dirp(\alpha) \times \gam(a, b)$ where $\dirp(\alpha)$ denotes the Dirichlet process distribution \citep{ferguson73} with base measure $\alpha$ and $\gam(a,b)$ denotes the gamma distribution with shape $a$ and rate $b$.


It is useful to recall two different characterizations of DP distributions, the original characterization by \citet{ferguson73} through a consistent system of Dirichlet distributions over measurable partitions and the later stick-breaking interpretation due to \citet{sethuraman94}. The first approach characterizes an $F \sim \dirp(\alpha)$, with $\alpha$ a finite measure on $\real^d$, as:
\begin{equation}
\label{dp1}
(F(B_1), \cdots, F(B_k)) \sim \textit{Dir}(\alpha(B_1), \cdots, \alpha(B_k)).
\end{equation} 
for any Borell measurable partition $B_1, \cdots, B_k$ of $\real^d$. The stick breaking characterization says an
\begin{equation}
\label{dp2}
F = \sum_{h = 1}^\infty \pi_h \delta_{Z_h}, \;\; \pi_h =  V_h \prod_{j < h}(1 - V_j),\;\;\delta_x = \mbox{ Dirach measure at }x,
\end{equation}
has a $\dirp(\alpha)$ distribution if $\{V_h, h \ge 1\}$ are independent $\textit{Be}(1, |\alpha|)$ random variables with $|\alpha| = \alpha(\real^d)$, $\{Z_h, h \ge 1\}$ are independently distributed according to the probability measure $\bar \alpha = \alpha / |\alpha|$ and these two sets of random variables are mutually independent. 

The base measure $\bar\alpha$ gives the mean of $F$, and also determines its support. The only assumptions we make on $\bar\alpha$ are that it admits a Lebesgue density that is strictly positive over the whole of $\real^d$ and that  for some constant $b_1$, $\bar \alpha([-a,a]^d) \lesssim \exp(-b_1 a^2)$, where $f(a) \lesssim g(a)$ means $f(a) \le K g(a)$ for all $a$, for some fixed constant $K$.


\section{Posterior convergence rates and adaptability}
\label{prelim}

Consider modeling $d$-variate measurements $X_1, X_2, \cdots$ as independent observations from a density $p$, which is assigned a prior distribution $\Pi$. Here $\Pi$ is a probability measure on the space $\scP$ of Lebesgue probability densities, equipped with the Borel $\sigma$-field under a metric $\rho$, usually taken to be the $L_1$ metric $\rho(p, q) = \|p - q\|_1 = \int_{\real^d} |p(x) - q(x)|dx$ or the Hellinger metric $\rho(p, q) = h(p, q) = [\int_{\real^d}\{p^{1/2}(x) - q^{1/2}(x)\}^2]^{1/2}$. 

Let $\Pi_n(\cdot|X_1, \cdots, X_n)$ denote the posterior distribution of $p$ based on the first $n$ measurements, defined for every measurable $B \subset \scP$ as
$$\Pi_n(B | X_1, \cdots, X_n) = \frac{\int_B \prod_{i = 1}^n p(X_i)\Pi(dp)}{\int_\scP \prod_{i = 1}^n p(X_i)\Pi(dp)}.$$
Let $\{\epsilon_n\}_{n \ge 1}$ be a sequence of positive numbers with $\lim_{n \to \infty} \epsilon_n = 0$. For any $p_0 \in \scP$ we say the posterior convergence rate at $p_0$ is (not slower than) $\epsilon_n$ if for some finite constant $M$
\beq
\label{def rate}
\lim_{n \to 0} \Pi(\{p: \rho(p_0, p) > M\epsilon_n\} | X_1, \cdots, X_n) = 0
\eeq
 almost surely whenever $X_1, X_2,\cdots$ are independent and identically distributed (iid) with density $p_0$.

Although \eqref{def rate} only establishes $\{\epsilon_n\}_{n \ge 1}$ as a bound on the convergence rate, it serves as a useful calibration of the method induced by $\Pi$ for classes of true densities $p_0$ for which optimal estimation rates are known. For example, for various classes of infinitely differentiable densities the optimal rate is known to be $n^{-1/2}(\log n) ^k$ for some $k \ge 0$ \citep{ibkhas83}, whereas for the class of compactly supported, twice continuously differentiable densities, the optimal rate is known to be $n^{-2/(4 + d)}$ \citep{huang04}. A method is considered adaptive if it provides convergence rates that are within a power of $\log n$ of these optimal rates. Along this line, we present the following results.

\begin{theorem}
\label{post rate}
Let $\Pi$ be the DP mixture prior of Section \ref{model}. 
\benum
\item If $p_0$ equals $p_{F_0, \sigma_0}$ for some probability measure $F_0$ on $\real^d$ and some $\sigma_0 > 0$, then \eqref{def rate} holds with $\epsilon_n = n^{-1/2} (\log n)^{(d + 1 + s )/2}$ for every $s > 0$. Such a $p_0$ will be called a super-smooth density. 
\item If $p_0$ is compactly supported and twice continuously differentiable then \eqref{def rate} holds with $\epsilon_n = n^{-2/(4 + d)}(\log n)^{(4d + 2)/(d + 4) + s}$ for every $s > 0$. Such a $p_0$ will be called an  ordinary-smooth density.
\eenum
\end{theorem}

These results are proved in Sections \ref{s sieve} and \ref{s conc}. The main tool needed to  establish \eqref{def rate} is a set of sufficient conditions proposed in \citet[Theorem 2.1]{ghosal&etal00}. We present here a slightly modified version adapted from \citet[Theorem 2.1]{ghosal&vandervaart01}. In the following, for any two probability densities $p$ and $q$ and any positive number $\epsilon$, we denote $K(p, q) = \int_{\real^d} p(x)\log \{p(x)/q(x)\} dx$, $V(p, q) = \int_{\real^d} p(x)[\log \{p(x)/q(x)\}]^2 dx$, $B(\epsilon; p) = \{q \in \scP: K(p, q) \le \epsilon^2, V(p, q) \le \epsilon^2\}$. For any $\scQ \subset \scP$, its $\epsilon$-covering number $N(\epsilon, \scQ, \rho)$ is defined to be the minimum number of balls of radius $\epsilon$ (in the metric $\rho$) needed to cover $\scQ$; with $\log N(\epsilon, \scQ, \rho)$ referred to as the $\epsilon$-entropy of $\scQ$.

\begin{theorem}
\label{suff}
Let $\rho$ be the Hellinger metric on $\scP$. Suppose for positive sequences $\tilde\epsilon_n$, $\bar\epsilon_n \to 0$ with $n\min(\tilde\epsilon^2_n,\bar \epsilon^2_n) \to \infty$, there exist positive constants $c_1, c_2, c_3, c_4$ and sets $\scP_n \subset \scP$, $n \ge 1$, such that  for all large $n$
\begin{align}
\log N(\bar \epsilon_n, \scP_n, \rho) & \le c_1 n\bar\epsilon_n^2, \label{gv1}\\
\Pi(\scP_n^c) & \le c_3 e^{-(c_2 + 4)n\tilde \epsilon_n^2} ,\label{gv2}\\
\Pi(B(\tilde\epsilon_n; p_0)) & \ge c_4e^{-c_2 n\tilde\epsilon_n^2} \label{gv3}.
\end{align}
Then \eqref{def rate} holds with $\epsilon_n = \max(\tilde\epsilon_n, \bar\epsilon_n)$.
\end{theorem}
\begin{remark}
If (\ref{def rate}) holds with $\rho = $ the Hellinger metric then it holds with $\rho =$ the $L_1$ metric, because for any two probability density $\|p - q\|_1 \le 2 h(p, q)$.
\end{remark}
It is common to call the sequence $\{\scP_n\}_{n \ge 1}$ a sieve on $\scP$. The first two conditions require existence of a low-entropy, high mass sieve. The third condition requires a quantitative bound on the thickness of the prior $\Pi$ at the true density $p_0$. We first take up the more challenging task of sieve construction for the DP mixture prior of Section \ref{model}, followed by prior thickness calculations.

\section{Sieve construction}
\label{s sieve}

\subsection{The basic construct}
The chief novelty of the sieve proposed in \citet{pati.etal.11} lies in exploiting the stick-breaking representation of a DP distribution. A high-mass, low-entropy subset of $\scP$ can be obtained by considering densities $p_{F,\sigma}$, with $F$ as given in \eqref{dp2} with limited tail mass $\sum_{h > H} \pi_h$. A precise statement is given below.

\begin{theorem}
\label{basic sieve}
Fix reals $\epsilon, a, \sigmalo > 0$ and integers $M, H \ge 1$. Define
\begin{equation} 
\label{eff}
\scQ = \bigg\{p_{F, \sigma}: F = \sum_{h = 1}^\infty \pi_h \delta_{z_h}:  z_h \in [-a,a]^d, h \le H; \sum_{h > H} \pi_h < \epsilon; 1 < \frac{\sigma}{\sigmalo} < (1 + \epsilon)^M\bigg\}. 
\end{equation} 
Then, for some positive constants $b_1, b_2$ and $b_3$, 
\begin{enumerate}
\item $\log N(\epsilon, \scQ, \rho) \lesssim dH \log\frac{a}{\sigmalo\epsilon} + H \log \frac{1}{\epsilon} + \log M$, where $\rho$ is either the $L_1$ or the Hellinger metric.
\item If $\Pi$ is the DP mixture prior of Section \ref{model}, then $\Pi(\scQ^c) \lesssim He^{-b_1a^2}  + e^{-b_2\sigmalo^{-d}} + \sigmalo^{-b_3d}(1 + \epsilon)^{-b_3dM} + \{(e|\alpha|/H)\log(1/\epsilon)\}^H$.
\end{enumerate}
\end{theorem}

\begin{proof}
Let $R^*$ be a $(\sigmalo\epsilon)$-net of $[-a,a]^d$ and let $S^*$ be an $\epsilon$-net of  the $H$-simplex $\S_H = \{p = (p_1,  \cdots, p_H): p_h \ge 0, \sum_h p_h = 1\}$. It is well known that the size of $R^*$ is $\lesssim \{a/(\sigmalo\epsilon)\}^d$ and that of $S^*$ is $\lesssim (1 / \epsilon)^H$. For any $p_{F, \sigma} \in \scQ$, with $F = \sum_{h = 1}^\infty z_h \delta_{z_h}$, find $z^*_1, \cdots, z^*_H \in R^*$, $\pi^* = (\pi^*_1, \cdots, \pi^*_H) \in S^*$ and $m^* \in \{1, \cdots, M\}$ such that
\begin{align}
& \max_{1 \le h \le H}\|z_h - z^*_h\| < \sigmalo \epsilon,\\
& \sum_{h = 1}^H |\tilde \pi_h - \pi^*_h| < \epsilon,\mbox{ where }\tilde \pi_h = \frac{\pi_h }{ 1-\sum_{l > H} \pi_l}, 1 \le h \le H, \mbox{ and}\\
& \sigma^* = \sigmalo(1 + \epsilon)^{m^*}\mbox{ satisfies }1 < \sigma/\sigma^* < 1 + \epsilon.
\end{align}
Then, with $F^* = \sum_{h = 1}^H \pi^*_h \delta_{z^*_h}$, we have,
\begin{align*} \|p_{F, \sigma} - p_{F^*, \sigma^*}\|_1 & \le  \|p_{F, \sigma} - p_{F, \sigma^*}\|_1 + \|p_{F, \sigma^*} - p_{F^*, \sigma^*}\|_1\\
& \le \frac{\sigma - \sigma}{\sigma^*} + \sum_{h > H} \pi_h + \sum_{h = 1}^H \pi_h \|\phi_{\sigma^*}(\cdot - z_h) - \phi_{\sigma^*}(\cdot - z^*_h)\|_1 + \sum_{h = 1}^H |\pi_h - \pi^*_h|.
\end{align*}
Each of the first three terms above is smaller than or equal to $\epsilon$. The last term is smaller than or equal to $(1 - \sum_{h > H}\pi_h)\sum_{h = 1}^H|\tilde \pi_h - \pi^*_h| + \sum_{h > H}\pi_h\sum_{h = 1}^H \pi^*_h \le 2\epsilon$
Thus a $5\epsilon$-net of $\scQ$, in the $L_1$ topology, can be constructed with $p^* = p_{F^*, \sigma^*}$ as above. The total number of such $p^*$ is $\lesssim (\frac{a}{\sigmalo\epsilon})^{dH}(\frac{1}{\epsilon})^HM$. This proves the first assertion of the theorem with $\rho = \|\cdot\|_1$; the constant multiplication by 5 can be absorbed in $\lesssim$ form of the bound. The same obtains for $\rho =$ the Hellinger metric because it is bounded by the square-root of the $L_1$ metric. 

Now with $\Pi$ denoting the DP mixture prior of Section \ref{model}, we have a stick-breaking representation of a random $p\sim \Pi$ given by $p = p_{F, \sigma} = \sum_{h = 1}^\infty \pi_h \phi_\sigma(\cdot - Z_h)$ with $\pi_h$ and $Z_h$ as described in \eqref{dp2} and the  paragraph that follows, and $\sigma^{-d} \sim \gam(a, b)$. Therefore,
\begin{equation} \Pi(\scQ^c) \le H \bar\alpha([-a,a]^d) + \Pr(\sigma^2 \not\in (\sigmalo^2, \sigmalo^2(1 + \epsilon)^{2M})) + \Pr\left(\sum_{h > H}\pi_h > \epsilon\right). 
\label{pc1}
\end{equation}
The first term is $\lesssim H\exp(-b_1a^2)$, by assumption on $\alpha$. The second term equals  $\Pr(\sigma^{-d} \ge \sigmalo^{-d}) + \Pr(\sigma^{-d} \le \sigmalo^{-d}(1 + \epsilon)^{-Md}) \lesssim \exp(-1/b_2 \sigmalo^{-d}) + (\sigmalo^d(1 + \epsilon)^{Md})^{-b_3}$ because $\sigma^{-d} \sim \gam(a, b)$. To bound the last term in \eqref{pc1}, note that $W = -\sum_{h = 1}^H\log(1 - V_h) \sim \textit{Ga}(H, |\alpha|)$, and therefore the last term equals
$$\Pr(W < \log(1 / \epsilon)) \le (|\alpha|\log \frac{1}{\epsilon})^H / \Gamma(H + 1) \le \bigg(\frac{e|\alpha|}{H}\log\frac{1}{\epsilon}\bigg)^H$$
by Stirling's formula. This proves the second assertion. 
\end{proof}

\subsection{Sieves for Theorem \ref{post rate}}
The subset $\scQ$ of Theorem \ref{basic sieve} can be easily adapted to form sieves targeted for different rates of convergence. Below we show this for the nearly parametric, super-smooth rate and also for the slower rates associated with H\"older classes of finitely differentiable functions. All this is done for any arbitrary dimension $d \ge 1$.

\begin{proposition}[Super-smooth rate]
\label{smooth sieve}
Fix any $s > 0$. For $\tilde \epsilon_n = n^{-1/2}(\log n)^{(d + 1)/2}$ and $\bar \epsilon_n = \tilde\epsilon_n (\log n)^{s/2}$,  there is a sequence of sets $\scP_n$ such that $\log N(\bar \epsilon_n, \scP_n, \rho) \lesssim n\bar\epsilon_n^2$ and $\Pi(\scP_n^c) \lesssim \exp(-cn\tilde\epsilon_n^2)$ for every $c > 0$, where $\rho$ is either the $L_1$ or the Hellinger metric. 
\end{proposition}

\begin{proof}
Let $\scP_n$ be defined as $\scQ$ of \eqref{eff} with $\epsilon = \bar\epsilon_n = n^{-1/2}(\log n)^{(d + 1 + s)/2}$, $H = n\bar\epsilon_n^2 / \log n = (\log n)^{d + s}$, and $M = a^2 = \sigmalo^{-d} = n$. Then, by Theorem \ref{basic sieve},
\begin{align*}
\log N(\bar \epsilon_n, \scP_n, \rho) & \lesssim d (\log n)^{d + s + 1} + (\log n)^{d + s + 1} + \log n\\
& \lesssim (\log n)^{d + s + 1} = n\bar \epsilon_n^2
\end{align*}
which proves the first assertion. Also,
\begin{equation}
\label{step1}
\Pi(\scP_n^c)  \lesssim (\log n)^{d + s}e^{-b_1n} + e^{-b_2 n} + n^{b_3}e^{-b_3 dn\log(1 + \bar \epsilon_n)} + (\log n)^{-(d + s - 1)(\log n)^{d + s}}.
\end{equation}
For any $c > 0$, the first, second and fourth terms on the right hand side of \eqref{step1} are clearly bounded by $C\exp(-c(\log n)^{d + s})$ for some constant $C$. The third term, too, is bounded by the same, possibly with different $C$ because $n\log(1 + \bar\epsilon_n) \gtrsim n\bar\epsilon_n^2 = (\log n)^s(\log n)^{d + 1} > c(\log n)^{d + s}$. And therefore $\Pi(\scP_n^c) \lesssim \exp(-c(\log n)^{d + 1})$. This proves the second assertion of the theorem.
\end{proof}

\begin{proposition}[H\"older-smooth sieve]
\label{ord sieve}
Fix any $\beta \in (0, 1/2)$, $q \ge 0$ and $s > 0$. For $\tilde \epsilon_n = n^{-\beta}(\log n)^{q}$, $\bar \epsilon_n = \epsilon_n (\log n)^s$,  there is a sequence of sets $\scP_n$ such that $\log N(\epsilon_n, \scP_n, \rho) \lesssim n\bar \epsilon_n^2$ and $\Pi(\scP_n^c) \lesssim \exp(-cn\tilde\epsilon_n^2)$ for every $c > 0$, where $\rho$ is either the $L_1$ or the Hellinger metric. 
\end{proposition}

\begin{proof}
Let $\scP_n$ be defined as on the right hand side of \eqref{eff} with $\epsilon = \bar \epsilon_n  = n^{-\beta}(\log n)^{q + s}$, $H = n\bar\epsilon_n^2 / \log n = n^{1 - 2\beta}(\log n)^{2(q + s) - 1}$, $M = a^2 = \sigmalo^{-d} = n$. Then by Theorem \ref{basic sieve}, $\log N(\epsilon_n, \scP_n, \rho) \lesssim n^{1 - 2\beta} (\log n)^{2(q + s)}$ and for every $c > 0$,
\begin{align*}
\label{step1b}
\Pi(\scP_n^c) & \lesssim n^{1 - 2\beta}(\log n)^{2(q + s) - 1}e^{-b_1n} + e^{-b_2 n} + n^{b_3}e^{-b_3 dn\log(1 + \bar\epsilon_n)} + n^{-(1 - 2\beta)n^{1 - 2\beta}(\log n)^{2(q + s)-1}}\\
& \lesssim e^{-(1 - 2\beta)n^{1 - 2\beta}(\log n)^{2(q + s)}} \lesssim e^{-cn^{1 - 2\beta}(\log n)^{2q}}.
\end{align*}

\end{proof}

The ordinary-smooth rate corresponds to $\beta = 2 / (4 + d)$, and more generally, a H\"older class of functions with continuous derivatives up to order $k$ corresponds to $\beta = k / (2k +d)$.

%
%

\section{Prior Thickness}
\label{s conc}

With sieve conditions (\ref{gv1}), (\ref{gv2}) taken care of, a proof of Theorem \ref{post rate} requires establishing the prior thickness property (\ref{gv3}) of $\Pi$ for each of the two classes of densities. Below we show that for a $p_0$ from either class, $\Pi(B(A\tilde \epsilon_n; p_0)) \gtrsim e^{-cn\tilde\epsilon_n^2}$ for some constants $A > 0, c > 0$ ,with $\tilde \epsilon_n$ as in Proposition \ref{smooth sieve} or Proposition \ref{ord sieve} as appropriate (with $\beta = 2 / (4 + d)$). This immediately leads to $\Pi(B(\tilde \epsilon_n; p_0)) \gtrsim e^{-c_2 n\tilde \epsilon_n^2}$ for some finite number $c_2 > 0$ and completes a proof of Theorem \ref{post rate}, with $\epsilon_n = \bar \epsilon_n$, because Propositions \ref{smooth sieve} and \ref{ord sieve} hold for all constants $c > 0$, including, $c = c_2 + 4$, as needed by Theorem \ref{suff}.

We will first tackle prior thickness at ordinary-smooth densities $p_0$ which present a bigger challenge than the super-smooth ones. Our proof closely follows the calculations presented in \citet{ghosal&vandervaart07} with some minor adaptation needed to handle higher dimensions. For this reason, most of the results are presented in the Appendix, with proofs given only for those where some adaptation is needed. However, we present the main argument below, because a similar argument presented in \citet[Section 9]{ghosal&vandervaart07} leaves some gaps (pun intended).

\begin{proposition}[Ordinary-smooth thickness]
\label{ordinarysmooth}
Suppose $p_0$ is compactly supported and
$$\int (\|\nabla p_0\| / p_0)^4 p_0d\lambda  < \infty,\;\; \int (\|\nabla^2 p_0\|_2 / p_0)^2 p_0 d\lambda < \infty,$$ where $\|A\|_2$ denotes the spectral norm of a matrix $A$. Then $
\Pi(B(A\tilde \varepsilon_n; p_0)) \gtrsim e^{-cn\tilde\varepsilon_n^2}$ with $\tilde \epsilon_n = n^{-2/(4 + d)}(\log n)^{(4d + 2)/(d + 4)}$ for some constants $A > 0, c > 0$.
\end{proposition}

\begin{proof}
Fix a $\sigma^2 \in \tilde\epsilon_n \{\log (1/\tilde\epsilon_n)\}^{-2} \cdot (1/2, 1)$. Find a $b > 1$ such that $\tilde\epsilon_n^b \{\log(1/\tilde\epsilon)\}^{9/4} \le \tilde\epsilon_n$. Let $P_0$ denote the probability measure associated with the density $p_0$. By Corollary \ref{cor disc approx}, there is a discrete probability measure $F_\sigma = \sum_{j = 1}^{N} p_j \delta_{z_j}$ with at most $N \lesssim \sigma^{-d} \log( 1 / \tilde\epsilon_n)^d$ support points in $[-a,a]^d$, with at least $\sigma \tilde\epsilon_n^{2b}$ separation between any $z_i \ne z_j$, such that 
$$\|p_{P_0, \sigma} - p_{F_\sigma,\sigma}\|_\infty \lesssim \varepsilon_n^{2b} / \sigma^{d + 1}~\mbox{and}~\|p_{P_0, \sigma} - p_{F_\sigma,\sigma}\|_1 \lesssim \varepsilon_n^{2b}\{\log(1/\varepsilon_n)\}^{1/2}.$$ Place disjoint balls $U_j$ with centers at $z_j$, $j = 1, \cdots, N$ with diameter $\sigma\varepsilon_n^{2b}$ each. Extend $\{U_1, \cdots, U_N\}$ to a partition $\{U_1, \cdots, U_K\}$ of $[-a, a]^d$ such that each $U_j$, $j = N + 1, \cdots, K$, has diameter smaller than or equal to $\sigma$. This can be done with $K \lesssim \sigma^{-d}\{\log(1/\tilde\epsilon_n)\}^d$. Further extend this to a partition $U_1, \cdots, U_M$ of $\real^d$ such that $(\sigma\tilde\epsilon_n^{2b})^d \lesssim \alpha(U_j) \le 1$ for all $j = 1, \cdots, M$. We can still have $M \lesssim \sigma^{-d}\{ \log(1 / \tilde\epsilon_n)\}^d \lesssim \tilde\epsilon_n^{-d/2}\{\log(1/\tilde\epsilon_n)\}^{2d}$. Define $p_j = 0$, $j = N + 1, \cdots, M$.  

Let $\scP_\sigma$ denote the set of probability measures $F$ on $\real^d$ with $\sum_{j = 1}^M |F(U_j) - p_j| \le 2\tilde\epsilon_n^{2db}$ and $\min_{1 \le j \le M} F(U_j) \ge \tilde\epsilon_n^{4db} / 2$. Then, by Lemma \ref{lem diff} (with $V_i = U_i$, $i = 1,\cdots, N$, $V_0 = \cup_{j > N} V_j$) for any $F \in \scP_\sigma$, $\|p_{F_\sigma,\sigma} - p_{F, \sigma}\|_\infty \lesssim \tilde\epsilon_n^{2b}/\sigma^d$,
$\|p_{F_\sigma,\sigma} - p_{F, \sigma}\|_1 \lesssim \tilde\epsilon_n^{2b}$ and hence, by Lemma \ref{lem dist} and Lemma \ref{lem p0}, 
\begin{align*}
h(p_0, p_{F, \sigma}) & \le h(p_0, p_{P_0, \sigma}) + h(p_{P_0, \sigma}, p_{F_\sigma, \sigma}) + h(p_{F_\sigma, \sigma}, p_{F, \sigma})\\
& \lesssim \sigma^2 + \tilde\epsilon_n^b \{\log(1/\tilde\epsilon_n)\}^{1/4} + \tilde\epsilon_n^b\\
& \lesssim \sigma^2 + \tilde\epsilon_n^b \{\log(1/\tilde\epsilon_n)\}^{1/4}
\end{align*}
Also, for any such $F$, for every $x \in [-a,a]^d$ with $J(x)$ denoting the $j \in \{1, \cdots, K\}$ such that $x \in U_j$,
$$p_{F,\sigma}(x) \ge \int_{\|z - x\| \le \sigma} \phi_\sigma(x - z) dF(z) \gtrsim \frac{1}{\sigma^d} \int_{\|x - z\| \le \sigma} dF(z) \ge \frac{1}{\sigma^d}F(U_{J(x)}) \gtrsim \frac{\tilde\epsilon_n^{4db}}{\sigma^d}$$
because, $U_{J(x)}$, with diameter no larger than $\sigma$, must be a subset of the ball of radius $\sigma$ around $x$. So $F \in \scP_\sigma$ implies $\log \|p_0 / p_{F, \sigma}\|_\infty \lesssim \log (1/\tilde\epsilon_n)$ and therefore, by Lemma \ref{lem dist}, $K(p_0, p_{F, \sigma}) \le A^2 \tilde\epsilon_n^2$ and $V(p_0 , p_{F, \sigma}) \le A^2 \tilde\epsilon_n^2$, for a universal constant $A > 0$ that does not depend on $\sigma$.

Note that $M\tilde\epsilon_n^{2db} \lesssim \tilde\epsilon_n^{2db - d/2}\{\log(1/\tilde\epsilon_n)\}^{2d} \le 1$ and for some large constant $a_1 > 0$, $\tilde\epsilon_n^{2db} \lesssim a_1 \{\min_{1 \le j \le M} \alpha(U_j)\}^{2/3}$. So, by Lemma \ref{lem dir}, $\Pr(F \in \scP_\sigma) \ge C \exp(-c M \log 1 / \tilde\epsilon_n) \gtrsim C\exp(-c\tilde\epsilon_n^{-d/2}\{\log(1/\tilde\epsilon_n)\}^{2d + 1})$, for some constants $C, c$ that depend on $\alpha(\real^d), a, d$ and $b$. Therefore, 
\begin{align*}
\Pi(B(A\tilde\epsilon_n; p_0)) & \gtrsim \exp(-c\tilde\epsilon_n^{-d/2}\{\log(1/\tilde\epsilon_n)\}^{2d + 1})\Pr(\sigma^2 \in \tilde\epsilon_n \{\log(1/\tilde\epsilon_n)\}^{-2}\cdot (1/2,1))\\
& = \exp(-c\tilde\epsilon_n^{-d/2}\{\log(1/\tilde\epsilon_n)\}^{2d + 1})\Pr(\sigma^d \in \tilde\epsilon_n^{d/2} \{\log(1/\tilde\epsilon_n)\}^{-d}\cdot (1/2^d,1))\\
& \gtrsim \exp(-c\tilde\epsilon_n^{-d/2}\{\log(1/\tilde\epsilon_n)\}^{2d + 1}) 
\end{align*}
because $\sigma^{-d}$ has a gamma distribution.

From this the result follows if $\tilde\epsilon_n^{-d/2}\{\log(1 / \tilde\epsilon_n)\}^{2d + 1} \le n\tilde\epsilon_n^2$. With $\tilde\epsilon_n = n^{-2 / (4 + d)}(\log n)^{q}$, we get $n\tilde\epsilon_n^2 = n^{d/(4 + d)}(\log n)^{2q}$ and $\tilde\epsilon_n^{-d/2}\{\log (1/\tilde\epsilon_n)\}^{2d + 1} < n^{d/(4 + d)}(\log n)^{2d + 1 - dq/2}$ and hence the condition is satisfied if $2d + 1 - dq / 2 \le 2q$, i.e., if $q \ge (4d + 2) / (d + 4)$. 
\end{proof}

Prior thickness calculation at a super-smooth $p_0$ follows along the same line, but is simpler because we can bypass the first step in the proof of Proposition \ref{ordinarysmooth} of approximating $p_0$ by a $p_{F,\sigma}$. In fact, this approximation is the main driver of the slower thickness rate $\tilde\epsilon_n$, the recent developments in \citet{kruijer.10} are about refining this approximation for densities that have higher order derivatives. 

\begin{proposition}[Super-smooth thickness]
\label{supersmooth}
If $p_0 = p_{F_0, \sigma_0}$ for some $F_0$ supported on $[-a,a]^d$, then $\Pi(B(A\tilde \varepsilon_n; p_0)) \gtrsim e^{-cn\tilde\varepsilon_n^2}$ with $\tilde \epsilon_n = n^{-1/2}(\log n)^{(d + 1)/2}$ for some constants $A, c > 0$.
\end{proposition}

\begin{proof}
Fix a $\sigma \in \sigma_0 \cdot (1 - \tilde\epsilon_n\{\log(1/\tilde\epsilon_n)\}^{-2}, 1)$. Fix $b > 1$ such that $\tilde\epsilon_n^b\{\log(1/\tilde\epsilon_n)\}^{9/4} \le \tilde\epsilon_n$. Construct $\scP_\sigma$ as before, but with $p_{F_0, \sigma}$ instead of $p_{P_0, \sigma}$. Because $\sigma$ is bounded from below by $\sigma_0/2$, this can be constructed with an $M \lesssim \{\log(1/\tilde\epsilon_n)\}^d$ and hence $\Pr(F \in \scP_\sigma) \gtrsim \exp(-c\{\log(1/\tilde\epsilon)\}^{d + 1})$ for some constant $c$. Note that 
\begin{align*}
\|p_0 - p_{F_\sigma, \sigma}\|_1 & \le \|p_0 - p_{F_0, \sigma}\|_1 + \|p_{F_0, \sigma} - p_{F_\sigma, \sigma}\|_1\le 1 - \sigma / \sigma_0 + \tilde\epsilon_n^{2b}\{\log(1/\tilde\epsilon_n)\}^{1/2}\\
& \le \tilde\epsilon_n \{\log(1 / \tilde\epsilon_n)\}^{-2}  + \tilde\epsilon_n^{2b}\{\log(1/\tilde\epsilon_n)\}^{1/2}
\end{align*}
and therefore, $F \in \scP_\sigma$ implies $K(p_0 , p_{F, \sigma}) \le A^2 \tilde\epsilon_n^2$ and $V(p_0, p_{F,\sigma})\le A^2 \tilde\epsilon_n^2$ for some universal constant $A > 0$ that does not depend on $\sigma$. Now, because $\Pr(\sigma \in \sigma_0(1 - \tilde\epsilon_n\{\log(1/\tilde\epsilon_n)\}^{-2}, 1)) \gtrsim \tilde\epsilon_n\{\log(1/\tilde\epsilon_n)\}^{-2} \gtrsim \exp(-\{\log(1/\tilde\epsilon_n)\}^{d + 1})$ we have $p_n \gtrsim \exp(-c\{\log(1/\tilde\epsilon_n)\}^{d + 1})$. 
From this the result follows if $\{\log(1/\tilde\epsilon_n)\}^{d + 1} \le n\tilde\epsilon_n^2$, which is satisfied with $\tilde\epsilon_n = n^{-1/2} (\log n)^q$ for $2q \ge d + 1$. 
\end{proof}

\appendix

\section{Appendix: Supporting results and proofs}

\begin{theorem}
\label{lem disc approx}
Let $P_0$ be a probability measure on $[-a,a]^d \subset \real^d$. For any $\varepsilon >0$ and $\sigma > 0$, there is a discrete probability measure $F_\sigma$ on $[-a,a]^d$ with at most $N_{\sigma, \varepsilon} = D[\{(a / \sigma) \vee 1\} \log(1/\varepsilon)]^d$ support points such that $\|p_{P_0, \sigma} - p_{F_\sigma, \sigma}\|_\infty \lesssim \varepsilon /\sigma^d$ and $\|p_{P_0, \sigma} - p_{F_\sigma, \sigma}\|_1 \lesssim \varepsilon \{\log (1/ \varepsilon)\}^{1/2}$, for some universal constant $D$.
\end{theorem}

\begin{proof} A proof of this result can be obtained through straightforward extensions of Lemma 2 of \citet{ghosal&vandervaart07} and Lemma 3.1 of \citet{ghosal&vandervaart01} to $d$ dimensions. The only subtlety lies in replacing display (3.9) of \citet{ghosal&vandervaart01} with 
\begin{equation}
\int z^l dF(z) = \int z^l dF'(z), \;\;l \in \{1, \cdots, 2k - 2\}^d
\label{new moment}
\end{equation}
where, for a $z = (z_1, \cdots, z_d) \in \real^d$ and a $l = (l_1, \cdots, l_d) \in \inte^d$, $z^l$ denotes $z_1^{l_1} z_2^{l_2}\cdots z_d^{l_d}$. For any probability distribution $F$ on $\real^d$, there exists a discrete distribution $F'$ with at most $\{2(k - 1)\}^d + 1$ support points, satisfying (\ref{new moment}). This power of $d$ propagate all through the require extensions and appears in $N_{\sigma, \epsilon}$ in the statement of the current theorem. 
\end{proof}

\begin{corollary}
\label{cor disc approx}
Let $P_0$ be a probability measure on $[-a,a]^d \subset \real^d$. For any $\varepsilon >0$ and $\sigma > 0$, there is a discrete probability measure $F^*_\sigma$ on $[-a,a]^d$ with at most $N_{\sigma, \varepsilon} = D[\{(a / \sigma) \vee 1\} \log(1/\varepsilon)]^d$ support points from the set $\{(n_1, \cdots, n_p) \sigma\varepsilon : n_i \in \inte, |n_i| < \lceil \frac{a}{\sigma\varepsilon}\rceil, i = 1, \cdots, p\}$
such that $\|p_{P_0, \sigma} - p_{F^*_\sigma, \sigma}\|_\infty \lesssim \varepsilon /\sigma^{d}$ and $\|p_{P_0, \sigma} - p_{F^*_\sigma, \sigma}\|_1 \lesssim \varepsilon \{\log (1/ \varepsilon)\}^{1/2}$. 
\end{corollary}

\begin{proof}
First get $F_{\sigma}$ as in Theorem \ref{lem disc approx} and then move each of its support points to the nearest point on the grid $\{(n_1, \cdots, n_p) \sigma\varepsilon : n_i \in \inte, |n_i| < \lceil \frac{a}{\sigma\varepsilon}\rceil, i = 1, \cdots, p\}$ to get $F^*_\sigma$. These moves cost at most a constant times $\epsilon^2/\sigma^d$ to the supremum norm distance and at most a constant times $\epsilon$ to the $L_1$ distance.
\end{proof}

\begin{lemma}
\label{lem p0}
Let $p_0$ be a twice continuously differentiable probability density on $\real^d$ and let $P_0$ denote the corresponding probability measure. If 
$$\int (\|\nabla p_0\| / p_0)^4 p_0d\lambda  < \infty \mbox{ and } \int (\|\nabla^2 p_0\|_2 / p_0)^2 p_0 d\lambda < \infty,$$ where $\|A\|_2$ denotes the spectral norm of a matrix $A$, then $h(p_0, p_{P_0,\sigma}) \lesssim \sigma^2$.
\end{lemma}
\begin{proof}
The proof below closely follows the proof of Lemma 4 in \cite{ghosal&vandervaart07} with some adaptation needed to handle $d > 1$. By the assumptions on $p_0$, $p_0$ and $\nabla p_0$ are uniformly bounded and hence $p_\sigma(x) := p_{P_0,\sigma}(x) = \int p_0(x - \sigma y) \phi(y) dy$ is twice continuously differentiable in $\sigma$ with derivatives $\dot p_\sigma(x)$ and $\ddot p_\sigma(x)$ given by
\begin{align*}
\dot p_\sigma(x) & = -\int y'\nabla p_0(x - \sigma y) \phi(y) dy\\
\ddot p_\sigma(x) & = \int \{y' \nabla^2 p_0(x - \sigma y) y \}\phi(y) dy.
\end{align*}
Using Taylor's theorem with the integral form of the remainder we have
$$p_\sigma^{1/2}(x) - p_0^{1/2}(x) = \sigma\frac{\dot p_0(x)}{2p_0^{1/2}(x)} + \frac{1}{2} \sigma^2 \int_0^1 \left( \frac{\ddot p_{s\sigma}(x)}{p_{s\sigma}^{1/2}(x)} - \frac{1}{2}\frac{\dot p_{s\sigma}^2(x)}{p_{s\sigma}^{3/2}(x)}\right)(1 - s) ds.$$
Because $\dot p_0(x) = -\int y' \nabla p_0(x) \phi(y) dy = 0$ for every $x$, we obtain
\begin{align*}
h^2(p_\sigma, p_0) & = \frac{1}{4} \sigma^4 \int \left( \int_0^1 \left( \frac{\ddot p_{s\sigma}(x)}{p_{s\sigma}^{1/2}(x)} - \frac{1}{2}\frac{\dot p_{s\sigma}^2(x)}{p_{s\sigma}^{3/2}(x)}\right)(1 - s) ds \right)^2 dx\\
& \le  \frac{1}{4} \sigma^4 \int_0^1 \int \left[ \left(\frac{\ddot p_{s\sigma}(x)}{p_{s\sigma}^{1/2}(x)}\right)^2 + \frac{1}{4}\left(\frac{\dot p_{s\sigma}^2(x)}{p_{s\sigma}^{3/2}(x)}\right)^2\right]dx \times (1 - s)^2 ds
\end{align*}
Now, for any $\sigma$, by the Cauchy-Schwarz inequality,
\begin{align*}
\ddot p_\sigma^2(x) & = \left(\frac{y'\nabla^2 p_0(x - \sigma y) y(y)}{p_0^{1/2}(x - \sigma y)} p_0^{1/2}(x - \sigma y) \phi(y) dy \right)^2\\
& \le \int \frac{(y' \nabla^2 p_0(x - \sigma y) y)^2}{p_0(x - \sigma y)}\phi(y) dy \times p_\sigma(x)\\
& \le \int \frac{\| \nabla^2 p_0(x - \sigma y) \|_2^2}{p_0(x - \sigma y)}\|y\|^4\phi(y) dy \times p_\sigma(x)
\end{align*}
and hence $\int (\ddot p_{s\sigma}(x) / p_{s\sigma}(x))^2 dx \le \int (\|\nabla^2 p_0\|_2 / p_0)^2 p_0 d\lambda \times \int \|y^4\|\phi(y) dy\lesssim 1$.

By H\"older's inequality with $p = 4$ and $q = 4/3$,
\begin{align*}
\dot  p^4_\sigma(x) & = \left(\int \frac{y'\nabla p_0(x - \sigma y)}{p_0^{3/4}(x - \sigma y)} p_0^{3/4}(x - \sigma y)\phi(y) dy\right)^4\\
& \le \left(\int \frac{\{y'\nabla p_0(x - \sigma y)\}^4}{p_0^3(x - \sigma y)} \phi(y)dy\right) \times \left(\int p_0(x - \sigma y)\phi(y) dy\right)^3\\
& \le \left(\int \frac{\|\nabla p_0(x - \sigma y)\|^4}{p_0^3(x - \sigma y)}\|y\|^4 \phi(y)dy\right) \times p_\sigma^3(x)
\end{align*}
and hence $\int (\dot p^2_{s\sigma}(x) / p^{3/2}_{s\sigma}(x))^2dx \le \int (\|\nabla p_0\| / p_0)^4 p_0d\lambda \times \int \|y\|^4\phi(y) dy \lesssim 1$.
\end{proof}

\begin{lemma}
\label{lem diff}
Let $V_0, V_1, \cdots, V_N$ be a partition of $\real^d$ and $F' = \sum_{j = 1}^N p_j \delta_{z_j}$ a probability measure on $\real^d$ with $z_j \in V_j$, $j = 1, \cdots, N$. Then, for any probability measure $F$ on $\real^d$, and any $\sigma > 0$,
\begin{align*}
\|p_{F, \sigma} - p_{F', \sigma}\|_\infty & \lesssim \frac{1}{\sigma^{d + 1}}\max_{1 \le j \le N} \diam(V_j) + \frac{1}{\sigma^d}\sum_{j = 1}^N |F(V_j) - p_j|\\
\|p_{F, \sigma} - p_{F', \sigma}\|_1 & \lesssim \frac{1}{\sigma}\max_{1 \le j \le N} \diam(V_j) + \sum_{j = 1}^N |F(V_j) - p_j|
\end{align*}
where $\diam(A) := \sup\{\|z_1 - z_2\| : z_1, z_2 \in A\}$ denotes the diameter of a set $A$.
\end{lemma}

\begin{proof} See the proof of Lemma 5 of \cite{ghosal&vandervaart07}. 
\end{proof}

\begin{lemma}[Lemma 10 of \cite{ghosal&vandervaart07}]
\label{lem dir}
Let $(X_1, \cdots, X_N) \sim {\text Dir}(\alpha_1, \cdots, \alpha_N)$, with $0 < \alpha_j \le 1$, $\sum_{i = 1}^N \alpha_j = m$. Fix $a > 0$, $b > 0$. Then, there exist constants $c$ and $C$ that only depend $a$, $b$ and $m$ such that for any $\varepsilon \in (0, \min(1/4, a \{\min_j\alpha_j\}^b, 1 / N))$, 
$$P\left(\sum_{j = 1}^N |X_j - p_j| \le 2\varepsilon, \min_j X_j \ge \varepsilon^2 / 2\right) \ge C \exp\left(-c N \log \frac{1}{\varepsilon}\right)$$
\end{lemma}

\begin{lemma}
\label{lem dist}
For every pair of probability densities $p$ and $q$,
\begin{align*}
P \log \frac{p}{q} & \lesssim h^2(p, q) \left(1 + \log\left\|\frac{p}{q}\right\|_\infty\right),\\[5pt]
P \left(\log \frac{p}{q}\right)^2 & \lesssim h^2(p, q) \left(1 +\log \left\|\frac{p}{q}\right\|_\infty\right)^2,\\[5pt]
\frac12\|p - q\|_1 &\le h(p, q) \le \|p - q\|^{1/2}_1.
\end{align*}
\end{lemma}
\begin{proof}
See Lemma 8 of \citet{ghosal&vandervaart07} for the first two inequalities. The last set is well known, \cite[e.g.,][page 212]{vandervaartbook}.
\end{proof}

\bibliographystyle{chicago}
\bibliography{/Users/suryatokdar/Desktop/Utils/TokdarReferences}
\end{document}